\documentclass[11pt,reqno]{amsart}
\usepackage{amsmath}
\usepackage{amsfonts}
\usepackage{amsthm}
\usepackage{amssymb, setspace}
\usepackage{mathtools}
\usepackage{lipsum}
\usepackage{cite}
\usepackage{color,graphicx,epsfig,geometry,fancyhdr,hyperref}
\usepackage[T1]{fontenc}
\usepackage{float}
\usepackage{subfig}
\usepackage{commath}
\usepackage{bbm}
\usepackage{mathrsfs,fleqn}
\usepackage{pgfplots}
\pgfplotsset{compat=newest}

\newtheorem{theorem}{Theorem}[section]

\newtheorem{lemma}{Lemma}[section]

\newcommand{\N}{\mathbb{N}}
\newcommand{\Z}{\mathbb{Z}}

\newcommand{\R}{\mathbb{R}}
\newcommand{\C}{\mathbb{C}}

\newcommand{\dnu}{\partial_\nu}
\newcommand{\dnux}{\partial_{\nu_x}}
\newcommand{\dnuy}{\partial_{\nu_y}}
\newcommand{\pdr}{\partial_r}
\newcommand{\dd}{\mathrm{d}}

\newcommand{\ov}{\overline}
\newcommand{\brac}[1]{\left\{#1\right\}}
\newcommand{\paren}[1]{\left( #1 \right) }
\newcommand{\exD}{\mathbb{R}^2 \setminus \overline{D}}

\usepackage{comment}
\usepackage{ulem}
\graphicspath{{paper-pics/}}

\begin{document}

\title[Direct sampling for recovering a clamped cavity]{Direct sampling for recovering a clamped cavity from the biharmonic far-field data}

\author{Isaac Harris}
\address{Department of Mathematics, Purdue University, West Lafayette, IN 47907, USA}
\email{harri814@purdue.edu}

\author{Heejin Lee}
\address{Department of Mathematics, Purdue University, West Lafayette, IN 47907, USA}
\email{lee4485@purdue.edu}

\author{Peijun Li}
\address{LSEC, ICMSEC, Academy of Mathematics and Systems Science, Chinese Academy of Sciences, Beijing 100190, China, and School of Mathematical Sciences, University of Chinese Academy of Sciences, Beijing 100049, China}
\email{lipeijun@lsec.cc.ac.cn}  

\thanks{The research of I. Harris and H. Lee is partially supported by the NSF DMS Grants 2107891 and 2208256.}



\begin{abstract}
This paper concerns the inverse shape problem of recovering an unknown clamped cavity embedded in a thin infinite plate. The model problem is assumed to be governed by the two-dimensional biharmonic wave equation in the frequency domain. Based on the far-field data, a resolution analysis is conducted for cavity recovery via the direct sampling method. The Funk--Hecke integral identity is employed to analyze the performance of two imaging functions. Our analysis demonstrates that the same imaging functions commonly used for acoustic inverse shape problems are applicable to the biharmonic wave context. This work presents the first extension of direct sampling methods to biharmonic waves using far-field data. Numerical examples are provided to illustrate the effectiveness of these imaging functions in recovering a clamped cavity.
\end{abstract}

\maketitle

\section{Introduction}

In this paper, we address an inverse shape problem arising from the study of wave propagation in a thin infinite elastic plate. For such a plate, the displacement is governed by the Kirchhoff--Love model, which corresponds to the biharmonic scattering problem in the frequency domain for the total displacement. We assume that the far-field data are measured and propose a numerical method to recover the unknown clamped cavity. The clamped boundary conditions imply that the total displacement has zero Cauchy data on the boundary of the cavity. This problem is analogous to the sound soft inverse shape problem in acoustic scattering, as discussed in \cite{Cakoni-Colton-book, DSMnf, Liu}. Such inverse problems have applications in non-destructive testing, where data collected from the exterior of an object are used to reconstruct its interior structure. These problems are relevant to various fields, including engineering, medical imaging, geophysics, and remote sensing. A similar inverse shape problem for biharmonic scattering has been explored in \cite{bih-opt}, using near-field data and an optimization strategy.

We develop a qualitative (i.e., non-iterative) method for recovering the cavity. For an in-depth discussion on qualitative methods applied to various inverse scattering problems, see \cite{Cakoni-Colton-book, ColtonKress, kirschbook}. Although extensive literature exists on applying qualitative methods to inverse acoustic, electromagnetic, and elastic problems, there is less theoretical and computational development specifically for biharmonic scattering. {\color{black} This makes the investigation mathematically interesting and challenging. This is due to the fact that well--known identities and theorems that are vital to the investigation may no longer be valid. Also, as we will see, by using the far-field data one inherently loss information about the solution outside the scatterer which is not the case for acoustic scattering.} Qualitative methods can be advantageous compared to nonlinear optimization approaches, which often require a prior information that may not be easily accessible, such as the number of regions to be identified, and can be computationally expensive. Therefore, we propose using the direct sampling method, which is computationally simple to implement and numerically stable. The proposed imaging functions involve computing the inner product or norm of the known data operator with a specific test function to obtain an analytic formula that can be exploited. However, it is worth noting that most qualitative methods rely on full aperture data for their theoretical justification.

There has been recent interest in studying inverse problems for biharmonic waves. In \cite{bih-invparam1,bih-invparam2}, the authors investigate the inverse parameter problem for the biharmonic operator in both bounded and unbounded domains. Additionally, \cite{bih-timedomain} addresses the biharmonic wave equation in the time domain and demonstrates that it is possible to uniquely determine the density and sources from the measured displacement and its Laplacian on the boundary. Other relevant works include \cite{bih-invsource,bih-invpotential}, where the authors examine the inverse random source and potential problems for the biharmonic wave equation in the frequency domain. This paper contributes to the limited literature on numerical methods for shape reconstruction involving biharmonic waves. {\color{black}Indeed, there are so few papers that provide the analytical framework for extending qualitative methods for biharmonic scattering problems. Also, the results obtained here are interesting and surprising that the same imaging functions for acoustic scattering works without alteration for biharmonic scattering. In many mathematical investigations, adding a single lower order term can dramatically alter the analysis as well as the results. The fact that at the moment this manuscript and \cite{LSM-BHclamped} are the only investigations for extending qualitative methods makes the results invaluable to researches in this field.}

The paper is outlined as follows. Section \ref{problem} presents the problem formulations, introducing both the direct and inverse problems. Section \ref{dsm} derives a factorization of the far-field operator, which is utilized in the resolution analysis of our imaging functions. Numerical examples demonstrating the applicability of the direct sampling method for cavity recovery are presented in Section \ref{numerics}. Finally, Section \ref{end} offers a brief discussion and conclusion of our results.

\section{Formulation of the Problem}\label{problem}

We consider the Kirchhoff--Love infinite plate problem involving a clamped cavity, where the cavity is represented  by a bounded region $D \subset \R^2$ with an analytic boundary $\partial D$. The cavity is exposed to a time-harmonic incident plane wave, denoted as $u^{\text{inc}}(x) = \text{e}^{\text{i}kx\cdot d}$, where $d$ is the incident direction on the unit circle $\mathbb{S}^1$. The total displacement field $u$ consists of the scattered field $u^{\text{scat}}$ and the incident field $u^{\text{inc}}$, i.e., $u \coloneqq u^{\text{scat}} + u^{\text{inc}}$. Moreover, the total field  satisfies the biharmonic scattering problem in $\mathbb{R}^2 \setminus \overline{D}$ for a fixed wave number $k>0$, described  by 
\begin{align}\label{biharmonic}
\Delta^2 u - k^4 u = 0 \quad  \text{in } \mathbb{R}^2\setminus \overline{D}, 
\end{align}
with the boundary conditions 
\begin{align}\label{cbc}
  u \big|_{\partial D}=0  \quad \text{ and } \quad \dnu{u} \big|_{\partial D}=0, 
\end{align}
where $\nu$ denotes the outward unit normal vector on the boundary $\partial D$. 

Since the incident field $u^{\text{inc}}$ satisfies 
$$ \Delta^2 u^{\text{inc}} - k^4 u^{\text{inc}} = 0 \quad  \text{in } \mathbb{R}^2,$$
the scattered field satisfies
$$ \Delta^2 u^{\text{scat}} - k^4 u^{\text{scat}} = 0  \quad  \text{in } \mathbb{R}^2\setminus \overline{D}.$$ 
To complete the system, we require that the scattered field $u^{\text{scat}}$ satisfies the Sommerfeld radiation condition at infinity (cf. \cite{poh-invsource}): 
\begin{align}\label{SRC}
\lim_{r \to \infty} \sqrt{r}(\pdr{u^{\text{scat}}} - \text{i} k u^{\text{scat}}) = 0,  \quad r=|x|,
\end{align}
which holds uniformly in $\hat{x}=x/|x|$.

It has been proven that the scattering problem given by \eqref{biharmonic}--\eqref{SRC} is well-posed. For a detailed analysis, we refer to  \cite{DongLi24} for the boundary integral equation approach and \cite{bih-var} for the variational approach. It has also been shown in \cite{bih-var,DongLi24} that the scattered field $u^{\text{scat}}$ can be expressed in terms of two auxiliary functions, $v_H$ and $v_M$, defined as
\begin{align}\label{vhvm}
v_H = -\frac{1}{2k^2}\big(\Delta u^{\text{scat}} - k^2 u^{\text{scat}}\big), \quad v_M = \frac{1}{2k^2}\big(\Delta u^{\text{scat}} + k^2 u^{\text{scat}}\big)
\end{align}
which implies that $u^{\text{scat}} = v_H + v_M$. 

It can be verified that $v_H$ satisfies the Helmholtz equation in $\exD$, while $v_M$ satisfies the modified Helmholtz equation in $\exD$. Therefore, the biharmonic wave scattering problem \eqref{biharmonic}--\eqref{SRC} is equivalent to the following problem:
\begin{align}
\Delta v_H + k^2 v_H = 0, \quad   \Delta v_M - k^2 v_M = 0 \quad &\text{in } \exD, \label{vhvmeq1} \\
v_H+v_M = -u^{\text{inc}}, \quad  \dnu(v_H+v_M) = -\dnu u^{\text{inc}} \quad &\text{on } \partial D,  \label{vhvmeq2}\\
\lim_{r \to \infty} \sqrt{r}(\pdr{v_H} - \text{i}kv_H) = 0, \quad  \lim_{r \to \infty} \sqrt{r}(\pdr{v_M} - \text{i}kv_M) = 0. &\label{vhvmeq3}
\end{align}
In our analysis, it is useful to note that the Sommerfeld radiation condition given in \eqref{vhvmeq3} can also be expressed as
\begin{align}\label{SRC1}
\pdr v_H - \text{i}k v_H = \mathcal{O}\big(r^{-3/2}\big), \quad  \pdr v_M - \text{i}k v_M = \mathcal{O}\big(r^{-3/2}\big),\quad r \to \infty. 
\end{align}

According to the analysis in \cite{DongLi24}, both $v_H$ and $v_M$ belong to $H^1_{loc} (\mathbb{R}^2\setminus \overline{D})$. It follows from the separation of variables that $v_H$ and $v_M$ admit the Fourier series expansions in 
$\mathbb{R}^2\setminus \overline{B_R}$:  
\begin{align*}
v_H (r,\theta) = \sum\limits_{|n|=0}^\infty  \frac{H^{(1)}_n(kr)}{H^{(1)}_n(kR)}  v^{(n)}_H \text{e}^{\text{i} n \theta},\quad  v_M (r,\theta) = \sum\limits_{|n|=0}^\infty   \frac{H^{(1)}_n(\text{i} kr)}{H^{(1)}_n(\text{i}kR)}  v^{(n)}_M \text{e}^{\text{i} n \theta}, 
\end{align*}
where $H^{(1)}_n$ denote the Hankel function of first kind of order $n\in \Z$, $v^{(n)}_H$ and $v^{(n)}_M$ are the Fourier coefficients of $v_H (R,\theta)$ and $v_M (R,\theta)$, respectively. Given this fact, the solution $v_M$ to the modified Helmholtz equation decays exponentially. Specifically, for all $n \in \Z$ , we have
$$| H^{(1)}_n(\text{i} kr)| = \mathcal{O}\big(r^{-1/2} \text{e}^{-kr}\big), \quad r \to \infty.$$
Additionally, using the recursive relationship for all $n \in \Z$, 
$$\frac{\text{d}}{\text{d}t} H^{(1)}_n (t) = \frac{1}{2} \big( H^{(1)}_{n-1} (t) + H^{(1)}_{n+1} (t)\big),$$
it follows that $\partial_r v_M$ decays exponentially as $r \to \infty$ (cf. \cite{bessel-webpage}). 

By Green's representation theorem, for $x\in \exD$, we have 
\begin{align}
\label{GrH}v_H(x) &= \int_{\partial D} \left(v_H(y) \dnuy \Phi_k(x,y) - \dnuy v_H(y) \Phi_k(x,y)\right) \dd s(y),  \\
\label{GrM}v_M(x) &= \int_{\partial D} \left( v_M(y) \dnuy \Phi_{\text{i}k}(x,y) - \dnuy v_M(y) \Phi_{\text{i}k}(x,y)\right) \dd s(y), 
\end{align}
where $\Phi_k(\cdot, \cdot)$ and $\Phi_{\text{i}k}(\cdot, \cdot)$ are the fundamental solutions of the Helmholtz equation and the modified Helmholtz equation in $\R^2$, respectively. Explicitly, the fundamental solution is given as
$\Phi_k(x,y) = \frac{\text{i}}{4} H^{(1)}_0(k|x-y|).$

Since $v_M$ decays exponentially and $v_H$ is a radiating solution to the Helmholtz equation, the scattered field $u^{\text{scat}} $ exhibits the asymptotic behavior
\begin{align}\label{asymp}
u^{\text{scat}} (x, d) = \frac{\text{e}^{\text{i}\pi/4}}{\sqrt{8\pi k}}\frac{\text{e}^{\text{i}k|x|}}{\sqrt{|x|}}\left( u^\infty(\hat{x},d) + \mathcal{O}(|x|^{-1})\right), \quad  |x| \to \infty.
\end{align}
Here, $u^\infty(\hat{x}, d)$ denotes the far-field pattern of $u^{\text{scat}} $. Combining \eqref{GrH} and \eqref{asymp} yields $u^\infty(\hat{x},d) = v^\infty_H (\hat{x},d)$ and 
\begin{align}\label{ffp1}
u^\infty(\hat{x},d) = \int_{\partial B_R}\Big( v_H(y) \dnuy \text{e}^{-\text{i}k\hat{x}\cdot y} -\dnuy v_H(y) \text{e}^{-\text{i}k\hat{x}\cdot y}\Big)\dd s(y),
\end{align}
where $\text{e}^{-\text{i}k\hat{x}\cdot y}$ represents the far-field pattern of the fundamental solution $\Phi_k (x,y)$ to the Helmholtz equation. 

Notice that the far-field data does not contain any information about the exponentially decaying component $v_M$  of the scattered field. Although some information is lost in the far field,  $u^\infty$ still uniquely determines the cavity $D$, as proven in \cite{DongLi23}. We define the far-field operator $F: L^2(\mathbb{S}^1) \to L^2(\mathbb{S}^1) $ as follows:
\begin{align}\label{ffop}
(Fg)(\hat{x}) = \int_{\mathbb{S}^1} u^\infty(\hat{x}, d) g(d) \, \dd s(d).
\end{align}

The goal is to develop a direct sampling method using the far-field operator to numerically recover the cavity $D$. To achieve this, the following section focuses on the necessary analysis to derive the proposed reconstruction method.

\section{Direct Sampling Method}\label{dsm}
In this section, we address the inverse shape problem of recovering the cavity $D$ from known biharmonic far-field data $u^\infty$ for all $\hat{x}$ and $d$ in $\mathbb{S}^1$. We aim to extend the applicability of direct sampling methods to biharmonic scattering data, marking the first exploration in this direction, to the best of our knowledge. Direct sampling methods have been studied for the analogous acoustic scattering problem, using far-field data in \cite{Liu} and near-field data in \cite{DSMnf}. For acoustic scattering problems (i.e., the Helmholtz equation), these methods are closely related to the factorization of the far-field operator; see \cite{kirschbook} for details on this factorization. We introduce the framework for the direct sampling method and define imaging functions for reconstructing $D$, which will be applied in subsequent numerical experiments.

As previously mentioned, the direct sampling method relies on the factorization of the far-field operator. We now proceed to derive a factorization of the far-field operator using boundary integral operators. Following the approach in \cite{DongLi24}, we assume that $v_H$ and $v_M$ can be represented as the double and single layer potentials with densities $\varphi_1$ and $\varphi_2$, respectively given as follows:
\begin{align*}
v_H (x)= \int_{\partial D} \dnuy \Phi_k(x,y) \varphi_1(y) \, \text{d}s(y), \quad v_M (x) =  \int_{\partial D} \Phi_{\text{i}k}(x,y) \varphi_2(y) \, \text{d}s(y). 
\end{align*}
Therefore, the scattered field is give by 
\begin{align}\label{scat-ansatz}
u^{\text{scat}}  (x) = \int_{\partial D} \dnuy \Phi_k(x,y) \varphi_1(y) \, \text{d}s(y) + \int_{\partial D} \Phi_{\text{i}k}(x,y) \varphi_2(y) \, \text{d}s(y), 
\end{align}
which satisfies the biharmonic wave equation 
$$ \Delta^2 u^{\text{scat}}  - k^4 u^{\text{scat}}  = 0  \quad  \text{in } \mathbb{R}^2\setminus \overline{D},$$ 
along with the radiation conditions \eqref{SRC}. 

The densities $\varphi_1$ and $\varphi_2$ can be determined by utilizing the boundary conditions in \eqref{vhvmeq2}. Applying the boundary conditions and jump relations, we obtain the system of boundary integral equations
\begin{align*}
\big({K}_k + {I} \big)\varphi_1 +  {S}_{\text{i}k}\varphi_2 &= -2 u^{\text{inc}}  \\
{T}_k\varphi_1  + \big({K}'_{\text{i}k}-{I} \big)\varphi_2 &= -2 \partial_\nu u^{\text{inc}}.
\end{align*}
For simplicity, the above linear system can be written as
\begin{align}\label{BIE}
\begin{pmatrix}
{K}_k + {I} & {S}_{\text{i}k} \\
{T}_k & {K}'_{\text{i}k}-{I}
\end{pmatrix}
\begin{pmatrix}
\varphi_1 \\ 
\varphi_2
\end{pmatrix}
=-2
\begin{pmatrix}
u^{\text{inc}} \\
\dnu{u^{\text{inc}}}
\end{pmatrix},
\end{align}
where the bounded integral operators are defined as follows (cf. \cite{int-equ-book,mclean}):
\begin{align*}
({S}_\tau \phi)(x) = 2\int_{\partial D} \Phi_\tau(x,y) \phi(y) \, \text{d}s(y),  \quad & \quad ({K}_\tau\phi)(x) = 2\int_{\partial D} \dnuy{\Phi_\tau(x,y)} \phi(y) \, \text{d}s(y), \\
 ({K}'_\tau)\phi(x) = 2\int_{\partial D} \dnux{\Phi_\tau(x,y)} \phi(y) \, \text{d}s(y), & \quad  ({T}_\tau)\phi(x) = 2\int_{\partial D} \dnux \dnuy{\Phi_\tau(x,y)} \phi(y)  \, \text{d}s(y), 
\end{align*}
for $x \in \partial D$, with the parameter $\tau$ equal to either $k$ or i$k$.

The matrix operator in \eqref{BIE} can be denoted as $\mathcal{Q} : \big( H^p (\partial D) \big)^2 \to  \big( H^{p+2 }(\partial D) \big)^2$ for any $p \geq 0$, where 
\begin{align*}
\mathcal{Q} =
\begin{pmatrix}
{K}_k + {I} & {S}_{\text{i}k} \\
{T}_k & {K}'_{\text{i}k}-{I}
\end{pmatrix}. 
\end{align*} 
It is shown in \cite{DongLi24} that, provided that the boundary $\partial D$ is analytic, the operator $\mathcal{Q} $ is a continuous bijection for any $p \geq 0$. In our analysis, we consider the case when $p=0$. Therefore, by the bounded inverse theorem, we have that 
$$ \mathcal{Q}^{-1}: \big( H^{2} (\partial D) \big)^2 \to  \big( L^{2}(\partial D) \big)^2 $$
is a bounded linear operator. Hence, the system of boundary integral equations in \eqref{BIE} can be solved as
\begin{align*}
\begin{pmatrix}
\varphi_1 \\ 
\varphi_2
\end{pmatrix}
=-2 \mathcal{Q} ^{-1}
\begin{pmatrix}
u^{\text{inc}} \\
\dnu{u^{\text{inc}}}
\end{pmatrix}.
\end{align*}

Now, we consider the injection $I_1: \big(L^2(\partial D)\big)^2 \to L^2(\partial D)$ given by  
\begin{align*}
I_1\begin{pmatrix}
\varphi_1 \\ 
\varphi_2
\end{pmatrix} = \varphi_1.
\end{align*}
This operator $I_1$ is linear and bounded, as follows from its definition. From \eqref{scat-ansatz}, we have
$$u^\infty (\hat{x},d) = \int_{\partial D} \dnuy \text{e}^{-\text{i}k\hat{x}\cdot y}  \varphi_1(y;d) \, \text{d}s(y), $$
where the dependence on the incident direction $d$ is made explicit. To proceed, notice that the above expression represents the far-field pattern of the double layer potential. With this in mind, we define
\begin{align*}
(DL_k)^\infty  \varphi_1 &=  \int_{\partial D} \dnuy \text{e}^{-\text{i}k\hat{x}\cdot y}  \varphi_1(y;d) \, \text{d}s(y)= -\text{i}k \int_{\partial D} \nu(y)\cdot \hat{x} \text{e}^{-\text{i}k\hat{x}\cdot y}  \varphi_1(y;d) \, \text{d}s(y), 
\end{align*}
which is a bounded linear operator mapping $ L^2(\partial D)$ to $L^2(\mathbb{S}^1)$.

Therefore, the far-field pattern for a clamped cavity $D$ can be expressed as 
\begin{align}\label{ffp}
u^\infty(\hat{x}, d) =-2 (DL_k)^\infty I_1 \mathcal{Q} ^{-1}(u^{\text{inc}}, \dnu{u^{\text{inc}}})^\top.
\end{align} 
Since the kernel of $(DL_k)^\infty$ and $(u^{\text{inc}}, \dnu{u^{\text{inc}}})^\top$ depend analytically on $\hat{x}$ and $d$, respectively, the far-field pattern $u^\infty$ depends analytically on $(\hat{x}, d) \in \mathbb{S}^1 \times \mathbb{S}^1$. Consequently, we obtain the following result. 

\begin{theorem}
The far-field operator $F: L^2(\mathbb{S}^1) \to L^2(\mathbb{S}^1)$, as defined by \eqref{ffop}, corresponding to the biharmonic clamped scattering problem \eqref{biharmonic}--\eqref{SRC}, is compact.
\end{theorem}

We now focus on deriving a factorization of the far-field operator $F$. This factorization is crucial for analyzing the direct sampling methods studied in this section. Once an appropriate factorization of $F$ is obtained, we can assess its applicability in reconstructing $D$ from the given data.

We define the operator $H: L^2(\mathbb{S}^1) \to \big( H^{2}(\partial D)\big)^2$ as follows: 
\begin{align*}
Hg = (v_g, \dnu{v_g})^{\top},
\end{align*}
where the Herglotz wave function $v_g$ is given by 
\begin{align}\label{hwfpq}
v_g (x) = \int_{\mathbb{S}^1} \text{e}^{\text{i}k x\cdot d} g(d)\, \text{d}s(d).
\end{align}

Notice that, similar to the case of acoustic scattering, the far-field operator can be factorized using \eqref{ffp} and the superposition principle. This follows directly from the linearity of our scattering problem \eqref{biharmonic} with respect to the incident field and its normal derivative on $\partial D$.

\begin{theorem}\label{ff-factor}
The far-field operator $F: L^2(\mathbb{S}^1) \to L^2(\mathbb{S}^1)$, defined by \eqref{ffop} for the biharmonic clamped scattering problem \eqref{biharmonic}--\eqref{SRC}, has the following factorization
\begin{align*}
Fg = -2 (DL_k)^\infty I_1 \mathcal{Q}^{-1}Hg. 
\end{align*}
\end{theorem}

To complete the derivation of our direct sampling method, we require one additional key ingredient. Specifically, we utilize the Funk--Hecke integral identity (cf. \cite{Liu}):
$$ \int_{\mathbb{S}^{1}} \text{e}^{\text{i}k(x-z)\cdot d} \, \text{d}s(d) = 2\pi J_0(k|x-z|),$$
where $J_0$ denotes the first kind Bessel function of order zero. This result can be readily derived from the Jacobi--Anger expansion for an incident plane wave (cf. \cite{ColtonKress}). If we let $\phi_z = \text{e}^{-\text{i}kz\cdot d}$, which is the far-field pattern of $\Phi_k (z,y)$, then $v_{\phi_z} = 2\pi J_0(k|x-z|)$. It is clear to note that the adjoint of $(DL_k)^\infty$ is given by 
$$ \left[(DL_k)^\infty\right]^* \phi = \text{i} k \int_{\mathbb{S}^{1}} \nu(y)\cdot d \text{e}^{\text{i}k y\cdot d}\phi(d)\, \text{d}s(d) = \dnuy  \int_{\mathbb{S}^{1}} \text{e}^{\text{i}ky\cdot d}  \phi(d)\, \text{d}s(d),$$
which implies that the  $\left[(DL_k)^\infty\right]^* \phi = \dnuy v_\phi$, i.e., the normal derivative of the Herglotz wave function. 

Combining the factorization and the above integral identities, we study two direct sampling imaging functions. We demonstrate that the following two imaging functions
\begin{align}\label{imagfunc}
z\longmapsto \left| (\phi_z, F\phi_z)_{L^2(\mathbb{S}^1)}\right|^\rho,\quad z \longmapsto \|F{\phi_z}\|^\rho_{L^2(\mathbb{S}^1)}
\end{align}
can effectively recover the cavity because they exhibit `rapid' decay as $\text{dist}(z,D) \rightarrow \infty$. The parameter $\rho>0$ can be selected to enhance the resolution of the imaging functions.

\begin{lemma}\label{farfieldineq}
The far-field operator $F: L^2(\mathbb{S}^1) \to L^2(\mathbb{S}^1)$, defined by \eqref{ffop} corresponding to the biharmonic clamped scattering problem \eqref{biharmonic}--\eqref{SRC}, satisfies 
$$ \left| (\phi_z, F\phi_z)_{L^2(\mathbb{S}^1)}\right| \leq C\| J_0(k|\cdot -z|)\|^2_{H^4(D)},$$
where $\phi_z= \mathrm{e}^{-\mathrm{i}kz \cdot d}$ is the far-field pattern of the fundamental solution $\Phi_k (z, {y})$. 
\end{lemma}

\begin{proof}
Recalling the factorization $F=-2 (DL_k)^\infty I_1 \mathcal{Q}^{-1}H$ given in Theorem \ref{ff-factor} and using the boundedness of the operators  
$$I_1 :  \big(L^2(\partial D)\big)^2 \to L^2(\partial D), \quad \mathcal{Q}^{-1} : \big( H^{2} (\partial D) \big)^2 \to  \big( L^2(\partial D) \big)^2, $$
we have 
\begin{align*}
\left| (\phi_z, F\phi_z)_{L^2(\mathbb{S}^1)}\right| &=  2 \left| \left( [ (DL_k)^\infty ]^* \phi_z, I_1 \mathcal{Q}^{-1} H\phi_z \right)_{L^2(\partial D)}\right| \\ 
& \leq C \left\| [(DL_k)^\infty]^*\phi_z \right\|_{L^2(\partial D)} \left\| H\phi_z \right\|_{H^2(\partial D)\times H^2(\partial D)}. 
\end{align*}
It follows from the Funk--Hecke integral identity that 
$$ H\phi_z =  2\pi \big(J_0(k|x-z|), \dnux{J_0(k|x-z|)} \big)^{\top},\quad  [(DL_k)^\infty]^*\phi_z =  2 \pi \dnux{J_0(k|x-z|)}.$$

It is clear to note that 
\begin{align*}
 \left\| [(DL_k)^\infty]^*\phi_z \right\|_{L^2(\partial D)} = 2 \pi  \|\dnu J_0(k|\cdot-z|) \|_{L^{2}(\partial D)} \leq C \|\dnu J_0(k|\cdot-z|) \|_{H^{1/2}(\partial D)}.
\end{align*}
By applying the trace theorem and utilizing the continuous embedding of $H^4(D)$ into $H^2(D)$, we obtain
$$\left\| [(DL_k)^\infty]^*\phi_z \right\|_{L^2(\partial D)} \leq C \| J_0(k|\cdot-z|) \|_{H^{4}(D)}.$$
Similarly, we have 
\begin{align*}
 \left\| H\phi_z \right\|^2_{H^2(\partial D)\times H^2(\partial D)} &=  2\pi  \|J_0(k|\cdot-z|)\|^2_{H^2(\partial D)} +  2\pi \|\dnu J_0(k|\cdot-z|)\|^2_{H^2(\partial D)}\\
 &\leq  C \left(  \|J_0(k|\cdot-z|)\|^2_{H^{7/2}(\partial D)}  + \|\dnu J_0(k|\cdot-z|)\|^2_{H^{5/2}(\partial D)} \right). 
\end{align*}
By applying the trace theorem again, we have
\begin{align*}
 \left\| H\phi_z \right\|_{H^2(\partial D)\times H^2(\partial D)} \leq C\|J_0(k|\cdot-z|)\|_{H^4(D)},
\end{align*}
which completes the proof.
\end{proof}

We now aim to use the estimate from Theorem \ref{farfieldineq} to show that the imaging functions in \eqref{imagfunc} decay as the sampling point $z \in \R^2$ moves away from the cavity. In previous works on acoustic scattering problems (see, for example, \cite{DSM-GH, DSMnf, DSM-HK, Zhang}), it was sufficient to estimate $\| J_0(k|\cdot -z|)\|^2_{H^1(D)}$. Using the asymptotic estimates 
$$J_0(t) = \frac{\cos t + \sin t}{\sqrt{\pi t}}  \left\{1 + \mathcal{O} (t^{-1}) \right\},\quad J_1(t) = \frac{\cos t - \sin t}{\sqrt{\pi t}}  \left\{-1 + \mathcal{O} (t^{-1}) \right\}$$ 
as $t \to \infty$, it follows that
$$\|J_0(k|\cdot-z|)\|_{H^1(D)} = \mathcal{O}\left(\mathrm{dist}(z, D)^{-1}\right),\quad \mathrm{dist}(z, D) \to \infty.$$
In our case, we need to establish this estimate for $\| J_0(k|\cdot -z|)\|^2_{H^4(D)}$. This estimate can be extended to hold for any $H^q(D)$-norm, where $q \in \N$.

\begin{lemma}\label{herglotz1}
For a fixed bounded domain $D \subset \R^2$, we have 
\begin{align*}
\| J_0(k|\cdot -z|)\|^2_{H^4(D)}= \mathcal{O}\big(\mathrm{dist}(z, D)^{-1}\big),\quad\mathrm{dist}(z, D) \to \infty. 
\end{align*}
\end{lemma}

\begin{proof}
It follows from the Funk--Hecke integral identity that 
$$v_{\phi_z} = 2\pi J_0(k|x-z|).$$ 
Therefore, it suffices to prove the result for the Herglotz wave function. For any $s,t \in \N \cup \{0\}$ such that $0\leq s+t \leq 4$, we have
\begin{align*}
 \partial_{x_1}^{s} \partial_{x_2}^{t} v_{\phi_z}  =  (\text{i}k)^{s+t }  \int_{\mathbb{S}^{1}} d_1^s d_2^t \text{e}^{\text{i}k(x-z)\cdot d} \, \text{d}s(d). 
 \end{align*}
Using $d=(\cos \phi , \sin \phi)$ and the Jacobi--Anger expansion 
$$\text{e}^{\text{i}k(x-z)\cdot d} = \sum\limits_{|n|=0}^\infty \text{i}^n J_n(k|x-z|)\text{e}^{\text{i}n(\theta -\phi)}, $$
where $\theta$ is the polar angle for $x-z$, we obtain 
\begin{align*}
 \partial_{x_1}^{s} \partial_{x_2}^{t} v_{\phi_z} =   (\text{i}k)^{s+t } \sum\limits_{|n|=0}^\infty  \text{i}^n   \int_{0}^{2\pi}   J_n(k|x-z|)\text{e}^{\text{i}n(\theta -\phi)}   (\cos{\phi})^{s} (\sin{\phi})^{t} \, \text{d}\phi,
 \end{align*}
which implies  
\begin{align*}
 \partial_{x_1}^{s} \partial_{x_2}^{t} v_{\phi_z} =  (\text{i}k)^{s+t } \sum\limits_{|n|=0}^\infty  \text{i}^n   J_n(k|x-z|) \text{e}^{\text{i}n\theta}\omega_n,
\end{align*}
where 
\begin{align*}
\omega_n = \frac{1 }{\text{i}^t 2^{s+t}} \int_{0}^{2\pi} \big(\text{e}^{\text{i}\phi} + \text{e}^{- \text{i}\phi}\big)^{s} \big(\text{e}^{\text{i}\phi} - \text{e}^{- \text{i}\phi} \big)^{t} \text{e}^{- \text{i}n\phi}\, \text{d} \phi.
\end{align*}

Note that, due to the orthogonality of $\text{e}^{\text{i}n\phi}$ for $n \in \Z$ over the interval $[ 0 , 2\pi ]$, it follows that $\omega_n =0$ for all $|n|>s+t$. 
This implies that 
 \begin{align*}
 \partial_{x_1}^{s} \partial_{x_2}^{t} v_{\phi_z} = \sum_{|n|=0}^{s+t} c_n J_n(k|x-z|)  \text{e}^{\text{i}n\theta}, 
\end{align*}
where $c_n =  (\text{i}k)^{s+t } \text{i}^n \omega_n$.

Given the asymptotic behavior of the Bessel function (cf. \cite{ColtonKress,Zhang}),
\begin{align*}
J_n(t) = \mathcal{O}(t^{-1/2}) \quad \text{as }  \quad t \to \infty,
\end{align*}
and the fact that the partial derivatives $ \partial_{x_1}^{s} \partial_{x_2}^{t} v_{\phi_z} $ form a finite sum of Bessel functions, we obtain
\begin{align*}
\|\partial_{x_1}^{s} \partial_{x_2}^{t} v_{\phi_z}\|^2_{L^2(D)}=  \mathcal{O}\big(\mathrm{dist}(z, D)^{-1}\big),  \quad\mathrm{dist}(z, D) \to \infty, 
\end{align*}
which proves the claim. 
\end{proof}

With this result, along with Lemma \ref{farfieldineq}, we state our first resolution analysis result for the imaging function
$$ z \longmapsto \left| (\phi_z, F\phi_z)_{L^2(\mathbb{S}^1)}\right|^\rho, \quad \rho>0.$$
The following result shows that the this imaging function decays rapidly as the sampling point $z$ moves away from the cavity $D$. 

\begin{theorem}\label{farfieldineq1}
The far-field operator $F: L^2(\mathbb{S}^1) \to L^2(\mathbb{S}^1)$, defined by \eqref{ffop} corresponding to the biharmonic clamped scattering problem \eqref{biharmonic}--\eqref{SRC}, satisfies 
$$ \left| (\phi_z, F\phi_z)_{L^2(\mathbb{S}^1)}\right| = \mathcal{O}\big(\mathrm{dist}(z, D)^{-1}\big),\quad \mathrm{dist}(z, D) \to \infty,$$
where $\phi_z= \mathrm{e}^{-\mathrm{i}kz \cdot d}$ is the far-field pattern of the fundamental solution $\Phi_k (z, {y})$. 
\end{theorem}

Next, we turn our attention to the other imaging function under consideration, given by
$$ z \longmapsto \|F{\phi_z}\|^\rho_{L^2(\mathbb{S}^1)},\quad \rho>0.$$
We demonstrate that this imaging function also decays rapidly as the sampling point $z$ moves away from the cavity $D$. To achieve this, we extend a well-known result from the case of acoustic sound-soft scatterers to our case of a biharmonic clamped cavity. This result is commonly used in the study of the factorization method for recovering sound soft scatterers. One major advantage of direct sampling methods over the more analytically rigorous factorization method is their stability with respect to noisy data. In fact, these methods can be used to recover scatterers even when the data contains a significant amount of noise. The factorization method for recovering a biharmonic clamped cavity has not yet been studied.

\begin{theorem}\label{ffoeq}
Let $F: L^2(\mathbb{S}^1) \to L^2(\mathbb{S}^1)$ be the far-field operator defined by \eqref{ffop}, corresponding to the biharmonic clamped scattering problem \eqref{biharmonic}--\eqref{SRC}. Then, there holds 
\begin{align*}
F-F^* = \frac{\textup{i}}{4\pi}F^*F, 
\end{align*}
where $F^*$ is the adjoint operator of $F$.
\end{theorem}

\begin{proof}
To prove the claim, let $g, h \in L^2(\mathbb{S}^1)$, and consider the Herglotz wave functions with kernels $g$ and $h$, respectively. Let $P, Q$ represent the total fields associated with the biharmonic clamped scattering problem \eqref{biharmonic}--\eqref{SRC}, corresponding to the incident waves $v_g$ and $v_h$, respectively. As in the previous section, the corresponding scattered fields have the decomposition $P^{\text{scat}} = P_H + P_M$ and $Q^{\text{scat}} = Q_H+Q_M$,  as defined by \eqref{vhvm}, and satisfy \eqref{vhvmeq1}--\eqref{vhvmeq3} with the respective incident waves.
By the superposition principle, $Fg$ and $Fh$ correspond to the far-field patterns for the scattered fields $P^{\text{scat}}$ and $Q^{\text{scat}}$, respectively. Then, by appealing to the exponential decay of the solutions to the modified Helmholtz equation, we obtain 
$$Fg =  P_H^\infty,\quad  Fh=Q_H^\infty.$$
From Green's identifies in $B_R \setminus \overline{D}$, we have 
\begin{align}\label{eqPQ}
0  = \int_{B_R \setminus \overline{D}} \paren{ \Delta^2P\ov{Q}  -P\Delta^2\ov{Q}} \dd x 
 = \int_{\partial B_R} \paren{ \dnu P \Delta\ov{Q} - P \dnu\Delta\ov{Q} + \dnu \Delta P \ov{Q} - \Delta P\dnu\ov{Q} }\dd s, 
\end{align}
where the boundary integral over $\partial D$ vanishes due to the clamped boundary conditions.

Notice that since the Herglotz wave function solves the Helmholtz equation in the entire plane, along with \eqref{vhvmeq1}, we obtain that
\begin{align*}
\Delta P= k^2(P_M - P_H - v_g) \quad \text{and} \quad \Delta Q = k^2(Q_M-Q_H-v_h).
\end{align*}
With this, we see that the first two terms in \eqref{eqPQ} can be written as 
\begin{multline*}
\int_{\partial B_R} \paren{ \dnu P \Delta\ov{Q}-P \dnu\Delta\ov{Q}  }\dd s \\
= k^2\int_{\partial B_R} \left[  \dnu\paren{P_H+P_M+v_g} \ov{\paren{Q_M-Q_H-v_h}} - \paren{P_H+P_M+v_g} \dnu\ov{\paren{Q_M-Q_H-v_h}} \right]\dd s.
\end{multline*}
Since the terms $P_M, \dnu P_M, Q_M$, and $\dnu Q_M$ decay exponentially as $R \to \infty$, we can simplify the previous equation as follows
\begin{multline*}
\int_{\partial B_R} \paren{ \dnu P \Delta\ov{Q}-P \dnu\Delta\ov{Q}  }\dd s\\ = 
k^2\int_{\partial B_R} \left[  \dnu\paren{P_H+v_g} \ov{\paren{-Q_H-v_h}}- \paren{P_H+v_g} \dnu\ov{\paren{-Q_H-v_h}}  \right]\dd s + o(1)
\end{multline*}
as $R \to \infty$. It can be verified from the Sommerfeld radiation condition \eqref{SRC1} and the asymptotic behavior \eqref{asymp} that
\begin{align*}
k^2\int_{\partial B_R}  \paren{ P_H\dnu \overline{Q}_H - \dnu P_H \overline{Q}_H} \dd s \to -\frac{\text{i}k^2}{4\pi} \int_{\mathbb{S}^1}P_H^\infty\overline{Q}_H^\infty \dd s =  -\frac{\text{i}k^2}{4\pi}\paren{Fg, Fh}_{L^2(\mathbb{S}^1)}
\end{align*}
as $R \to \infty$. By appealing to \eqref{hwfpq} we have that  
\begin{align*}
k^2 \int_{\partial B_R} \paren{ v_g\dnu \overline{Q}_H - \dnu v_g \overline{Q}_H} \dd s
 & = k^2 \int_{\mathbb{S}^1} g(\hat{x}) \left[ \int_{\partial B_R} \paren{\text{e}^{\text{i} k\hat{x}\cdot y} \dnu \overline{Q}_H(y) - \dnu \text{e}^{\text{i}k \hat{x} \cdot y} \overline{Q}_H(y)} \dd s(y) \right] \dd s(\hat{x})\\
 & =  -k^2 \int_{\mathbb{S}^1} g(\hat{x}) \overline{Q}_H^\infty(\hat{x}) \dd s(\hat{x}) \\
 &= - k^2\paren{g, Fh}_{L^2(\mathbb{S}^1)}, 
\end{align*}
where the conjugate of the far-field pattern of $Q_H$ from \eqref{ffp1} is used to obtain
\begin{align*}
k^2\int_{\partial B_R} \paren{P_H\dnu \overline{v}_h - \dnu P_H \overline{v}_h} \dd s = k^2 \paren{Fg, h}_{L^2(\mathbb{S}^1)}.
\end{align*}
Notice that the term
\begin{align*}
\int_{\partial B_R} \paren{v_g\dnu \overline{v}_h - \dnu v_g\overline{v}_h} \dd s =0 
\end{align*}
due to the fact that, both $v_g$ and $v_h$ satisfy the Helmholtz equation in the entire plane. Therefore, we have that 
\begin{align*}
\int_{\partial B_R} \paren{ \dnu P \Delta\ov{Q}-P \dnu\Delta\ov{Q}  }\dd s
\to 
-\frac{\text{i}k^2}{4\pi}\paren{Fg, Fh}_{L^2(\mathbb{S}^1)}  - k^2\paren{g, Fh}_{L^2(\mathbb{S}^1)} + k^2 \paren{Fg, h}_{L^2(\mathbb{S}^1)}
\end{align*}
as $R \to \infty$.

Now, we consider the other in \eqref{eqPQ}. Arguing is a similar manner, we see that
\begin{align*}
&\int_{\partial B_R} \paren{ \dnu \Delta P \ov{Q}-\Delta P \dnu \ov{Q}  }\dd s \\
&= k^2\int_{\partial B_R} \left[  \dnu\paren{-P_H+P_M-v_g} \ov{\paren{Q_M+Q_H+v_h}} - \paren{-P_H+P_M-v_g} \dnu\ov{\paren{Q_M+Q_H+v_h}} \right]\dd s \\
&= k^2\int_{\partial B_R} \left[  \dnu\paren{-P_H-v_g} \ov{\paren{Q_H+v_h}}- \paren{-P_H-v_g} \dnu \ov{\paren{Q_H +v_h}}  \right]\dd s + o(1)
\end{align*}
as $R \to \infty$. Note that we have used the exponential decay of the solutions to the modified Helmholtz equation as in the analysis of the first two terms. With this we see that 
\begin{multline*}
\int_{\partial B_R} \paren{ \dnu \Delta P \ov{Q}-\Delta P \dnu \ov{Q}  }\dd s \\
= k^2\int_{\partial B_R} \left[  \dnu\paren{P_H+v_g} \ov{\paren{-Q_H-v_h}}- \paren{P_H+v_g} \dnu \ov{\paren{-Q_H-v_h}}  \right]\dd s + o(1)
\end{multline*}
as $R \to \infty$. This implies letting $R \to \infty$ we have 
\begin{align*}
\int_{\partial B_R} \paren{ \dnu \Delta P \ov{Q}-\Delta P \dnu \ov{Q}  }\dd s \to 
-\frac{\text{i}k^2}{4\pi}\paren{Fg, Fh}_{L^2(\mathbb{S}^1)}  - k^2\paren{g, Fh}_{L^2(\mathbb{S}^1)} + k^2 \paren{Fg, h}_{L^2(\mathbb{S}^1)}
\end{align*}
i.e. the first two terms and the last two terms have the same limiting value. Combining the two limiting results we have that 
\begin{align*}
&0 =\int_{\partial B_R} \paren{ \dnu P \Delta\ov{Q} - P \dnu\Delta\ov{Q} + \dnu \Delta P \ov{Q} - \Delta P\dnu\ov{Q} }\dd s \\
&\hspace{1.5in}\to  -\frac{\text{i}k^2}{2\pi}\paren{Fg, Fh}_{L^2(\mathbb{S}^1)}  - 2k^2\paren{g, Fh}_{L^2(\mathbb{S}^1)} + 2k^2 \paren{Fg, h}_{L^2(\mathbb{S}^1)},
\end{align*}
as $R \to \infty$. Therefore, we have obtained that 
\begin{align*}
\frac{\text{i}}{4\pi}\paren{Fg, Fh}_{L^2(\mathbb{S}^1)} = \paren{Fg, h}_{L^2(\mathbb{S}^1)}-\paren{g, Fh}_{L^2(\mathbb{S}^1)}
\end{align*}
which completes the proof.
\end{proof}

We are now prepared to derive the resolution analysis for the second imaging function. As we will demonstrate, both of the aforementioned functions can be used to recover the cavity $D$. These imaging functions are numerically straightforward to compute and stable with respect to noisy data. This provides two practical methods for numerically recovering the unknown region $D$ from the measured far-field data.

\begin{theorem}\label{imgfunc2}
The far-field operator $F: L^2(\mathbb{S}^1) \to L^2(\mathbb{S}^1)$, defined by \eqref{ffop} corresponding to the biharmonic clamped scattering problem \eqref{biharmonic}--\eqref{SRC}, satisfies 
$$ \|F\phi_z\|_{L^2(\mathbb{S}^1)} = \mathcal{O}\big(\mathrm{dist}(z, D)^{-1/2}\big),\quad \mathrm{dist}(z, D) \to \infty,$$
where $\phi_z= \mathrm{e}^{-\mathrm{i}kz \cdot d}$ is the far-field pattern of the fundamental solution $\Phi_k (z, {y})$. 
\end{theorem}

\begin{proof}
First, recall that the imaginary part of the far-field operator is defined as
$$\Im\{F\} = \frac{1}{2\text{i}}(F-F^*).$$
Using the result from Theorem  \ref{ffoeq}, we show that the two imaging functions bound each other from above and below. Specifically, we establish that
$\|F\phi_z\|_{L^2(\mathbb{S})}$ and $ \big| (\phi_z, F\phi_z)_{L^2(\mathbb{S}^1)}\big|$ are related in a sense of equivalence.

Indeed, it is clear that 
\begin{align*}
 \big| (\phi_z, F\phi_z)_{L^2(\mathbb{S}^1)}\big| \leq \|\phi_z\|_{L^2(\mathbb{S}^1)}\|F\phi_z\|_{L^2(\mathbb{S}^1)} = \sqrt{2\pi}\|F\phi_z\|_{L^2(\mathbb{S}^1)}.
\end{align*}
Conversely, we have
\begin{align*}
 \left| (\phi_z, F\phi_z)_{L^2(\mathbb{S}^1)}\right| \geq  \left| \Im\brac{{(\phi_z, F\phi_z)_{L^2(\mathbb{S}^1)}}}\right|  =  \left| (\phi_z, \Im\brac{{F}}\phi_z)_{L^2(\mathbb{S}^1)}\right| =\left|  \frac{1}{2\text{i}}\paren{\phi_z, (F-F^*)\phi_z}_{L^2(\mathbb{S}^1)} \right|.
\end{align*}
From Theorem \ref{ffoeq}, we obtain
\begin{align*}
 \big| (\phi_z, F\phi_z)_{L^2(\mathbb{S}^1)}\big|  \geq \frac{1}{8\pi} \paren{\phi_z, F^*F \phi_z}_{L^2(\mathbb{S}^1)} = \frac{1}{8\pi} \|F\phi_z\|^2_{L^2(\mathbb{S}^1)}.
\end{align*}
This establishes the claim by utilizing the decay result in Theorem \ref{farfieldineq1}.
\end{proof}

From the proof of Theorem \ref{imgfunc2}, it is clear that the two imaging functions under consideration are equivalent in the sense that there exist constants constants $c_1,c_2>0$ such that 
$$c_1 \|F\phi_z\|^2_{L^2(\mathbb{S}^1)} \leq  \left| (\phi_z, F\phi_z)_{L^2(\mathbb{S}^1)}\right| \leq c_2 \|F\phi_z\|_{L^2(\mathbb{S}^1)}. $$ 
Therefore, we expect both functions to yield similar reconstructions of the cavity $D$. The primary difference is that the imaging function defined via the inner product has a faster decay rate compared to the imaging functional defined via the norm. As we will demonstrate in our numerical experiments, both functions provide good reconstructions, and the difference is not detectable in practice, provided that $\rho$  is chosen such that both functions exhibit the same decay rate as $\mathrm{dist}(z, D) \to \infty$.

\section{Numerical Validation}\label{numerics}
In this section, we provide numerical examples to illustrate the effectiveness of the two imaging functions discussed in the previous section. We use synthetically generated far-field data for these examples. To obtain this data, we employ the numerical method described in \cite{DongLi24} to solve the boundary integral equation (BIE)  \eqref{BIE} for the densities $\varphi_1$ and $\varphi_2$ where the integral equations are discretized using the standard method as described in \cite{DongLi24} and \cite{int-equ-book}.

In our examples, we define the matrix
$$\textbf{F}=\Big[u^\infty(\hat{x}_i,d_j)\Big]^{64}_{i,j=1}, $$ 
where
\begin{align*}
 u^\infty (\hat{x},d) =-\text{i}k \int_{\partial D} \nu(y)\cdot \hat{x} \text{e}^{-\text{i}k\hat{x}\cdot y}  \varphi_1(y;d) \, \text{d}s(y). 
\end{align*}
Here, we discretize $\mathbb{S}^1$ by setting    
$$\hat{x}_i=d_i= \big(\cos(\theta_i),\sin(\theta_i)\big)^\top,\quad \theta_i=2\pi(i-1)/64,\quad i=1,\dots,64.$$ 
Thus, the multi-static far-field matrix $\textbf{F} \in \C^{64 \times 64}$ corresponds to $64$ incident and observation directions.

To simulate measurement error in the data, random noise is introduced into the multi-static far-field matrix  $\mathbf{F}$, yielding
$$\mathbf{F}^{\delta}=\Big[ \mathbf{F}(i,j) \big(1+\delta\mathbf{R}(i,j) \big)\Big]^{64}_{i,j=1},$$
where the error matrix $\mathbf{R} \in \C^{64 \times 64}$ consists of random values within the interval $[-1, 1]$, with $0<\delta < 1$ representing the relative noise level added to the computed data. Consequently, the discretized imaging functions are given by
$$W_{\text{ip}}(z)=\big| (\boldsymbol{\phi}_z, \textbf{F}^{\delta}\boldsymbol{\phi}_z )_{\ell^2} \big|^\rho,\quad  W_{\text{norm}}(z)=\big\| \textbf{F}^{\delta} \boldsymbol{\phi}_z \big\|^\rho_{\ell^2},$$
where the vector $\boldsymbol{\phi}_z$ is computed as
$$\boldsymbol{\phi}_z=\big(\text{e}^{-\text{i}k z \cdot d_1},\dots, \text{e}^{-\text{i}k z \cdot d_{64}}\big)^\top,\quad z\in \mathbb{R}^2.$$

Recall that $\rho>0$ is a parameter that can be adjusted to refine the resolution of the reconstructions. According to Theorems \ref{farfieldineq1} and \ref{imgfunc2}, for any $\rho>0$, we have
$$W_{\text{ip}}(z)=  \mathcal{O}\big(\mathrm{dist}(z, D)^{-\rho}\big),\quad  W_{\text{norm}}(z)=\mathcal{O}\big(\mathrm{dist}(z, D)^{-\rho/2}\big)$$
as $\mathrm{dist}(z, D) \to \infty$. As demonstrated in our numerical examples, this indicates that both imaging functions can effectively recover the unknown cavity $D$.

We consider reconstructing two shapes: a peanut-shaped region and a star-shaped region, defined by
$$\partial D=r(\theta)\big(\cos(\theta),\sin(\theta)\big)^\top,\quad  0\leq \theta \leq 2 \pi,$$
where $r(\theta)=1.5 \big(0.5\sqrt{3\cos(\theta)^2+1}\big)$ is for the peanut shape and $r(\theta)=1.5\big(1+0.3\cos(4\theta)\big)$ is for the star shape. 

In these examples, we use a wave number $k=4$. The sampling region is $[-4,4]^2$, and we select 150$\times$150  equally spaced points within this region. We present contour plots of the imaging functions for both scatterers at different noise levels to assess the stability of our method. In all figures, the dotted line represents the boundary $\partial D$ of the cavity. The imaging functions are normalized in the sampling region to have a maximum value of one. We also examine cases with partial-aperture data, where complete access to the far-field pattern over the entire unit circle is not available.

\subsection{Example 1. A star-shaped cavity}

We present numerical reconstructions for the star-shaped cavity using both imaging functions. Figure \ref{recon1} shows the reconstruction obtained with $W_{\text{ip}}(z)$, while Figure \ref{recon2} displays the reconstruction obtained with $W_{\text{norm}}(z)$. We include examples with no error and with a small amount of error added to the data. In these examples, we set $\rho=4$ for both imaging functions.

\begin{figure}[H]
{\includegraphics[width=1\linewidth]{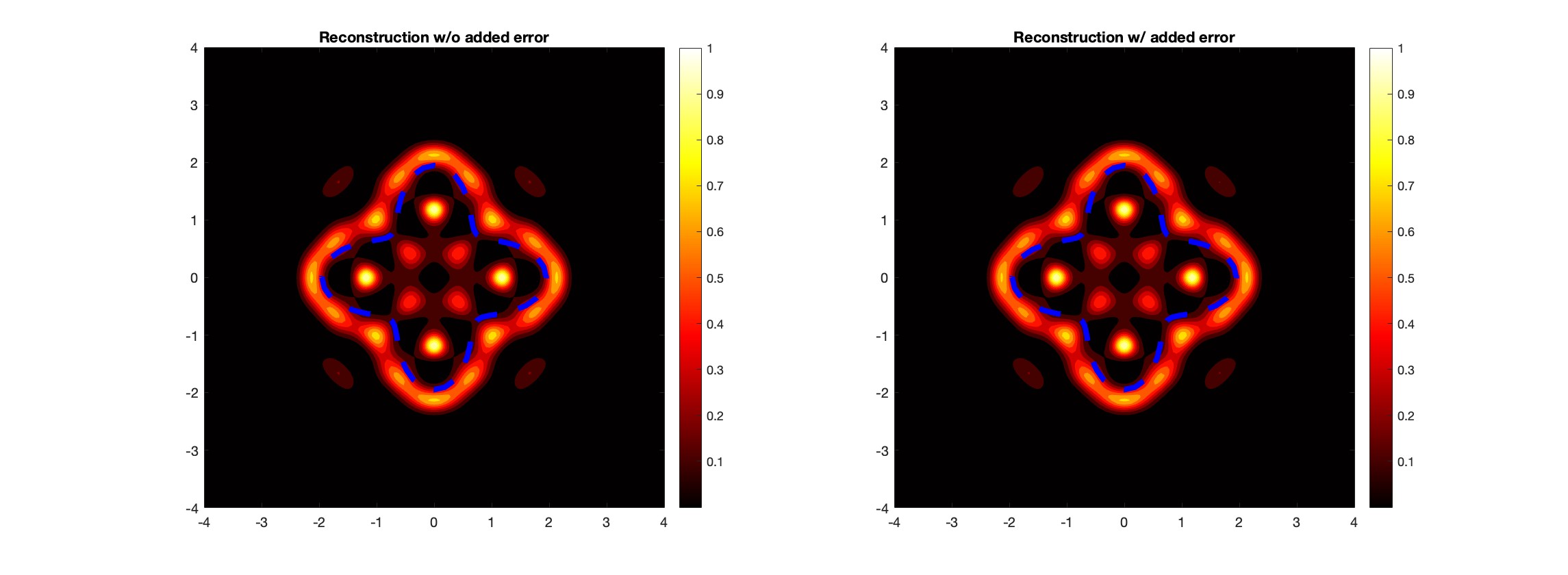}}  
\caption{The reconstruction of the star-shaped cavity by the imaging function $W_{\text{ip}}(z)$. (Left) the reconstruction with no error added to the far-field data; (Right) the reconstruction with $\delta=0.02$, which corresponds to a 2$\%$ noise level.}
\label{recon1} 
\end{figure}

\begin{figure}[H]
{\includegraphics[width=1\linewidth]{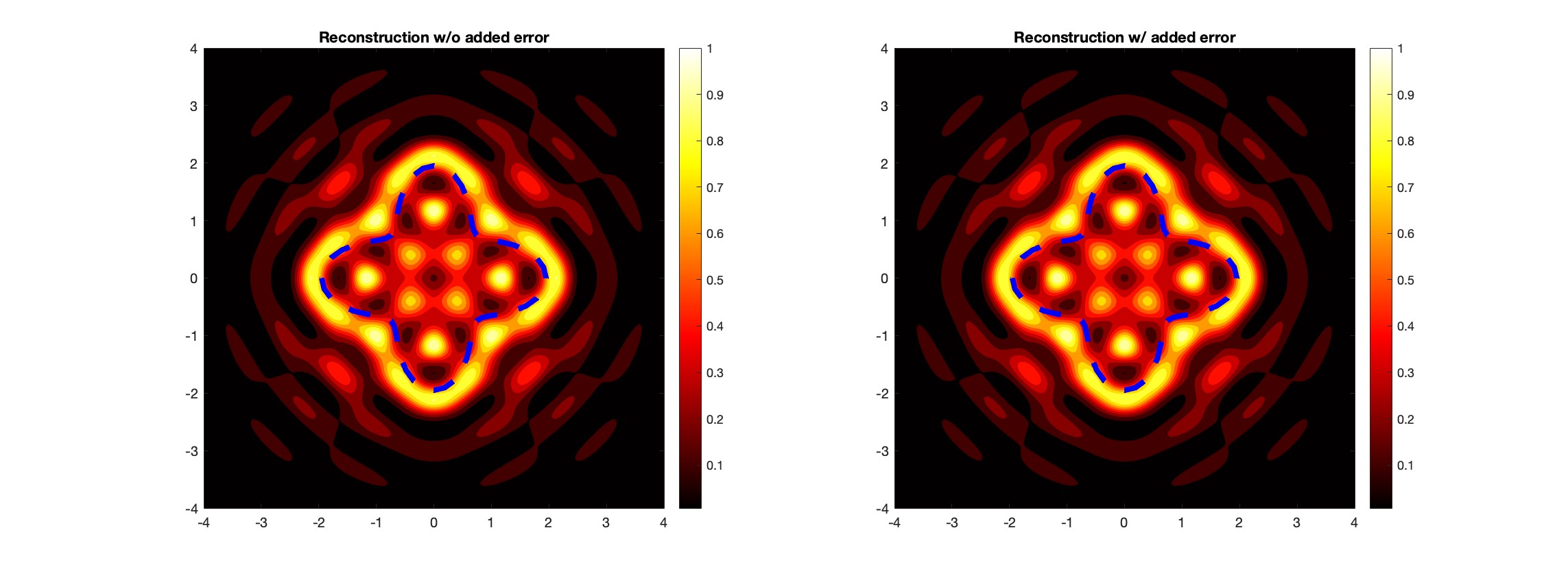}}  
\caption{The reconstruction of the star-shaped cavity by the imaging function $W_{\text{norm}}(z)$. (Left) the reconstruction with no error added to the far-field data; (Right) the reconstruction with $\delta=0.02$, which corresponds to a 2$\%$ noise level.}
\label{recon2} 
\end{figure}

Notice that $W_{\text{ip}}(z)$ appears to provide a sharper reconstruction, but this is primarily due to its faster decay rate. For $\rho=8$ in $W_{\text{norm}}(z)$, the reconstructions would be visually identical. This is illustrated in Figure \ref{recon3}, where no discernible difference is observed between the reconstructions using the two imaging functions with $\delta=0.02$, due to their equivalence.

\begin{figure}[H]
{\includegraphics[width=1\linewidth]{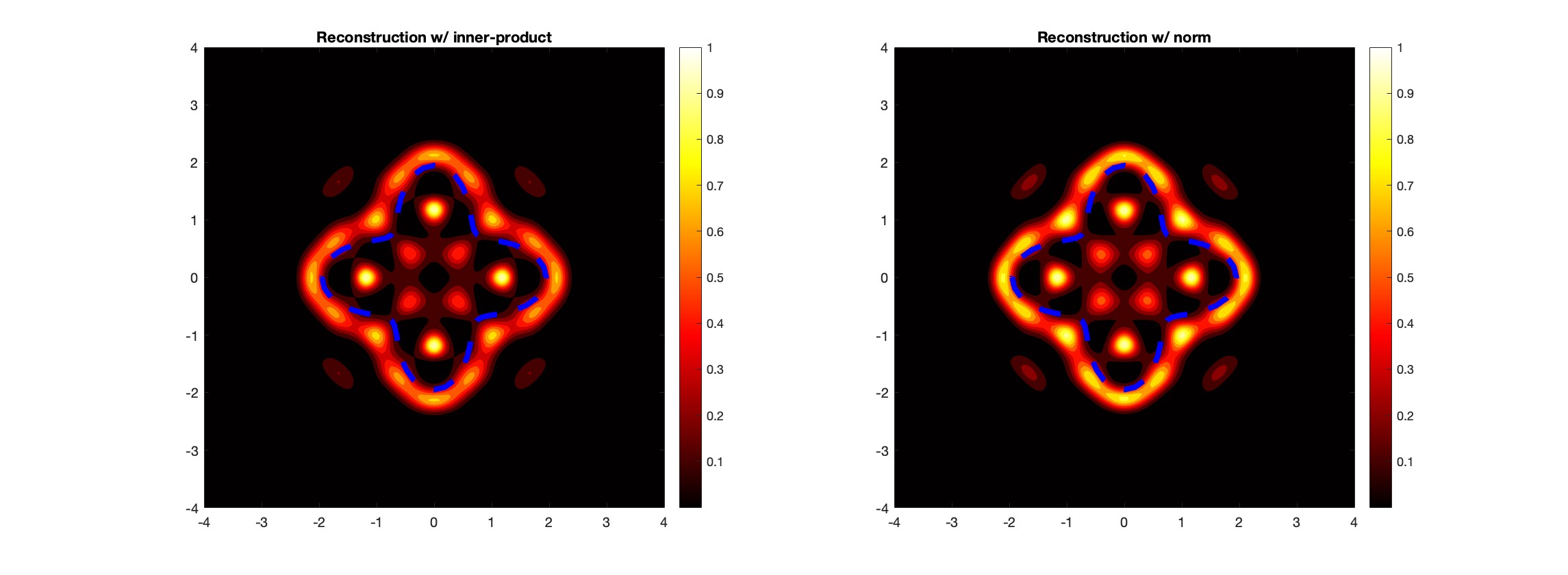}}  
\caption{The reconstruction of the star-shaped cavity. (Left) the reconstruction via $W_{\text{ip}}(z)$ with $\rho=4$; (Right) the reconstruction via $W_{\text{norm}}(z)$ with $\rho=8$.}
\label{recon3} 
\end{figure}

\subsection{Example 2. A peanut-shaped cavity}

We now present another numerical example using the peanut-shaped cavity with both imaging functions. Figure \ref{recon4} shows the reconstruction using $W_{\text{ip}}(z)$ with $\rho=4$ and $W_{\text{norm}}(z)$ with $\rho=8$. For these choices of $\rho$, both imaging functions exhibit the same decay rate, leading to identical reconstructions in the eyeball norm. In this example, $\delta=0.3$ is used, which represents a significant amount of noise in the data. Despite this, both imaging functions produce good reconstructions even with the presence of noisy data.

\begin{figure}[H]
{\includegraphics[width=1\linewidth]{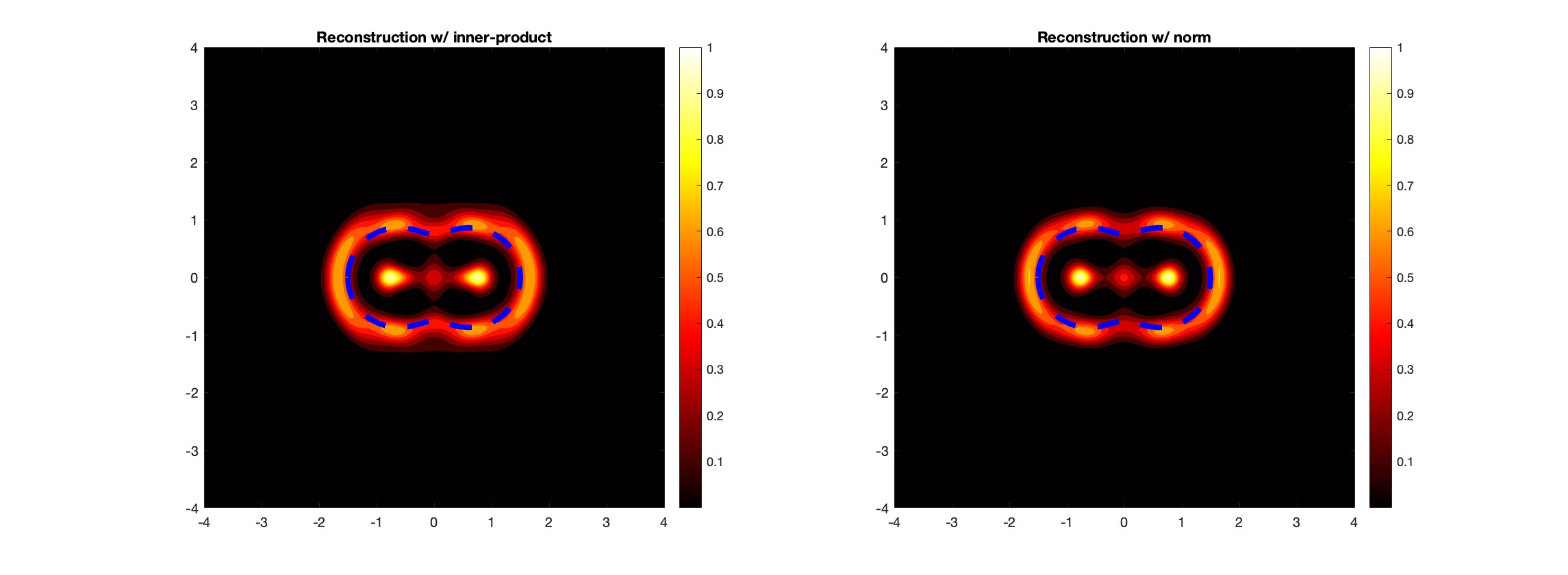}}  
\caption{The reconstruction of the peanut-shaped cavity with $\delta=0.3$, which corresponds to a 30$\%$ noise level. (Left) the reconstruction via $W_{\text{ip}}(z)$; (Right) the reconstruction via $W_{\text{norm}}(z)$.}
\label{recon4} 
\end{figure}

We also aim to assess the sensitivity of the reconstruction with respect to the parameter $\rho>0$. {\color{black}In Figure \ref{reconreviewer}, we present the numerical reconstructions using $W_{\text{ip}}(z)$ for $\rho=1$ and $\rho=1/2$. Then in Figure \ref{recon5}, we present the numerical reconstructions again using $W_{\text{ip}}(z)$ for $\rho=2$ and $\rho=8$.} The results demonstrate that larger values of $\rho$ yield sharper reconstructions of the boundary. Note that we test only one imaging function in this case due to their equivalence.
\begin{figure}[H]
{\includegraphics[width=1\linewidth]{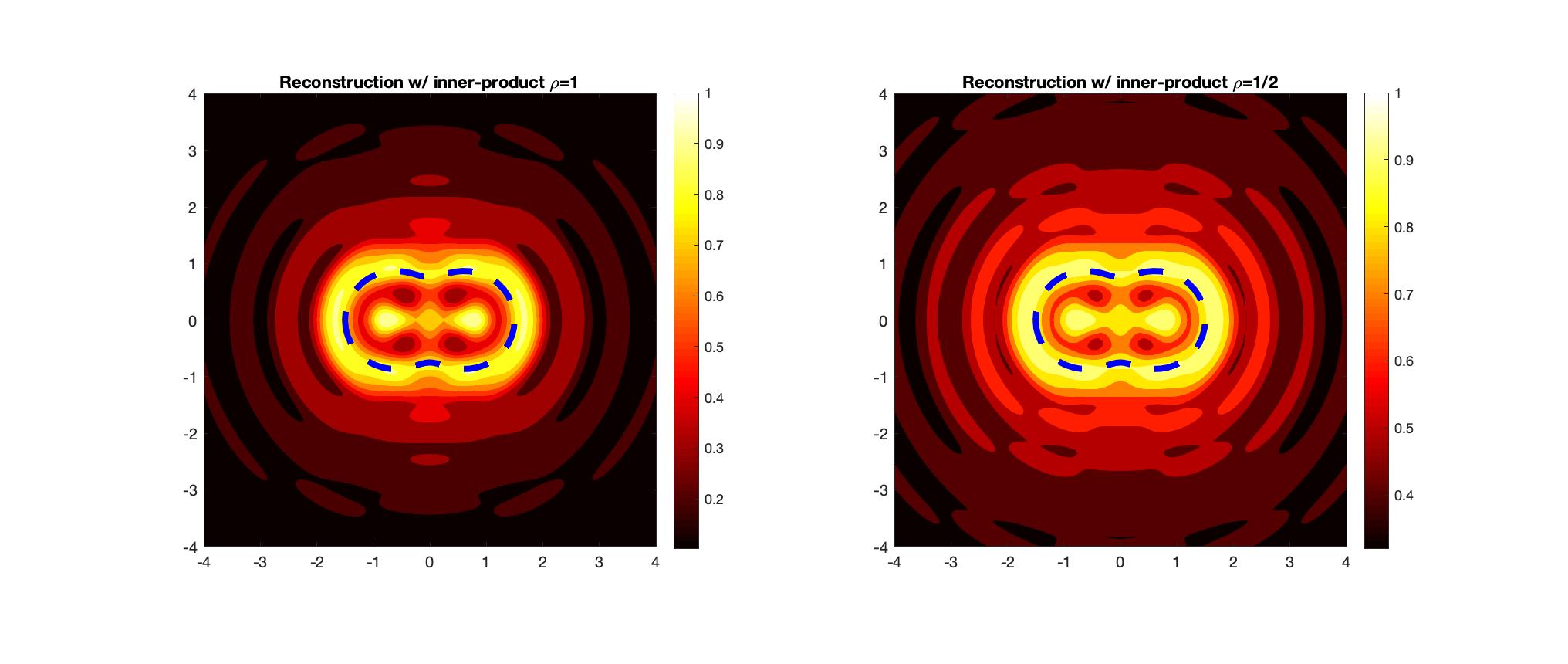}}  
\caption{The reconstruction of the peanut-shaped cavity with $\delta=0.1$, which corresponds to a 10$\%$ noise level via the imaging function $W_{\text{ip}}(z)$. (Left) the reconstruction with $\rho=1$; (Right) the reconstruction with $\rho=1/2$.}
\label{reconreviewer} 
\end{figure}

\begin{figure}[H]
{\includegraphics[width=1\linewidth]{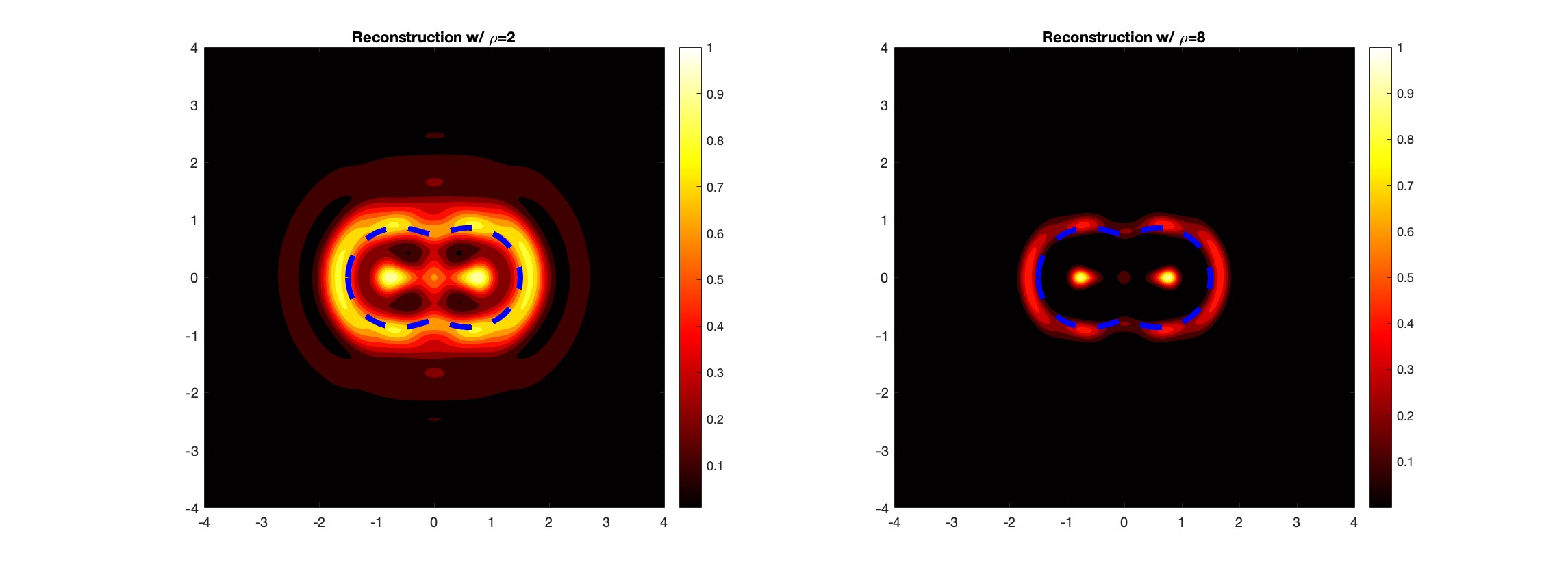}}  
\caption{The reconstruction of the peanut-shaped cavity with $\delta=0.1$, which corresponds to a 10$\%$ noise level via the imaging function $W_{\text{ip}}(z)$. (Left) the reconstruction with $\rho=2$; (Right) the reconstruction with $\rho=8$.}
\label{recon5} 
\end{figure}

\subsection{Example 3. Partial apertures}

We now investigate whether our imaging functions can recover cavities using partial-aperture data. In many applications, it is not always feasible to place sources and receivers around the entire cavity. To model partial aperture data, we set
$$\mathbf{F}^{\delta}(i,:) = \vec{0},\quad i=1, \cdots , 16,  \quad  \mathbf{F}^{\delta}(: \, ,j) = \vec{0},  \quad j=48, \cdots , 64. $$
This configuration simulates the absence of receivers in the first quadrant and sources in the fourth quadrant. In Figures \ref{recon6} and \ref{recon7}, we present the reconstructions of the two scatterers using this partial aperture data set with  $W_{\text{ip}}(z)$ and $\rho=4$. Despite the reduced data coverage compared to full aperture data, the imaging function still provides a good reconstruction of the cavities.

\begin{figure}[H]
{\includegraphics[width=1\linewidth]{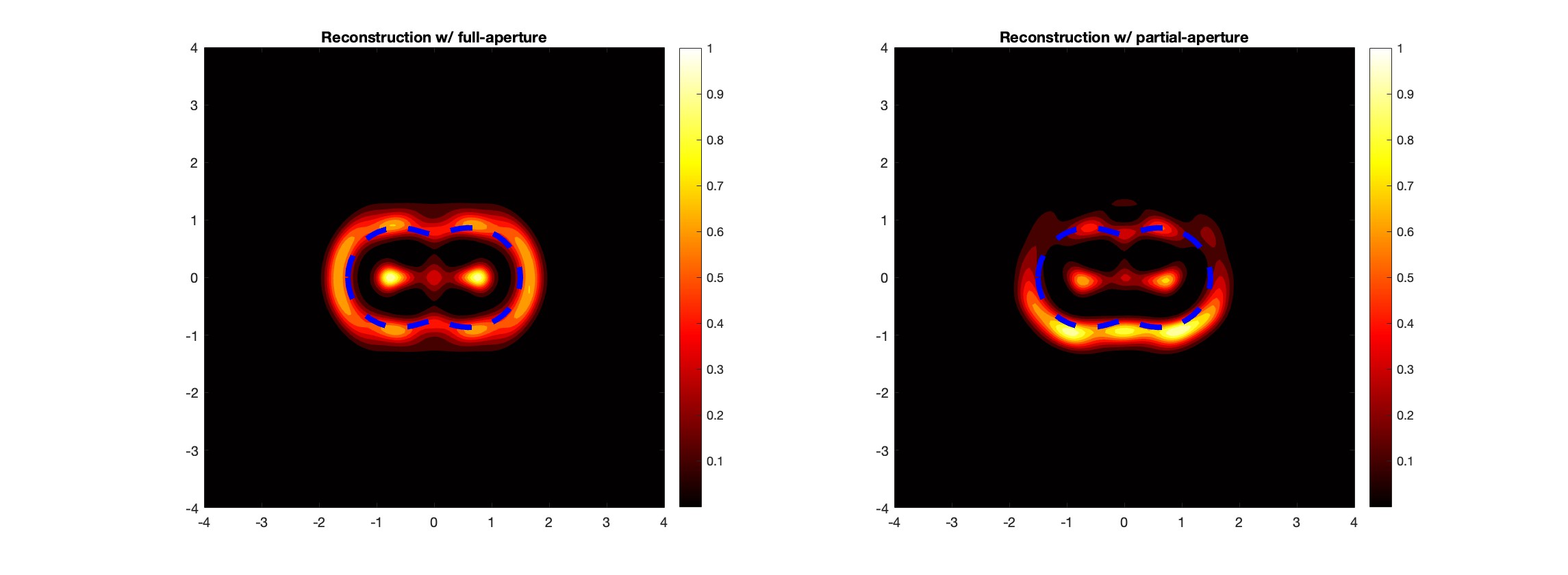}}  
\caption{The reconstruction of the peanut-shaped cavity with $\delta=0.1$, which corresponds to a 10$\%$ noise level via the imaging function $W_{\text{ip}}(z)$. (Left) the reconstruction with full aperture; (Right) the reconstruction with partial aperture.}
\label{recon6} 
\end{figure}

\begin{figure}[H]
{\includegraphics[width=1\linewidth]{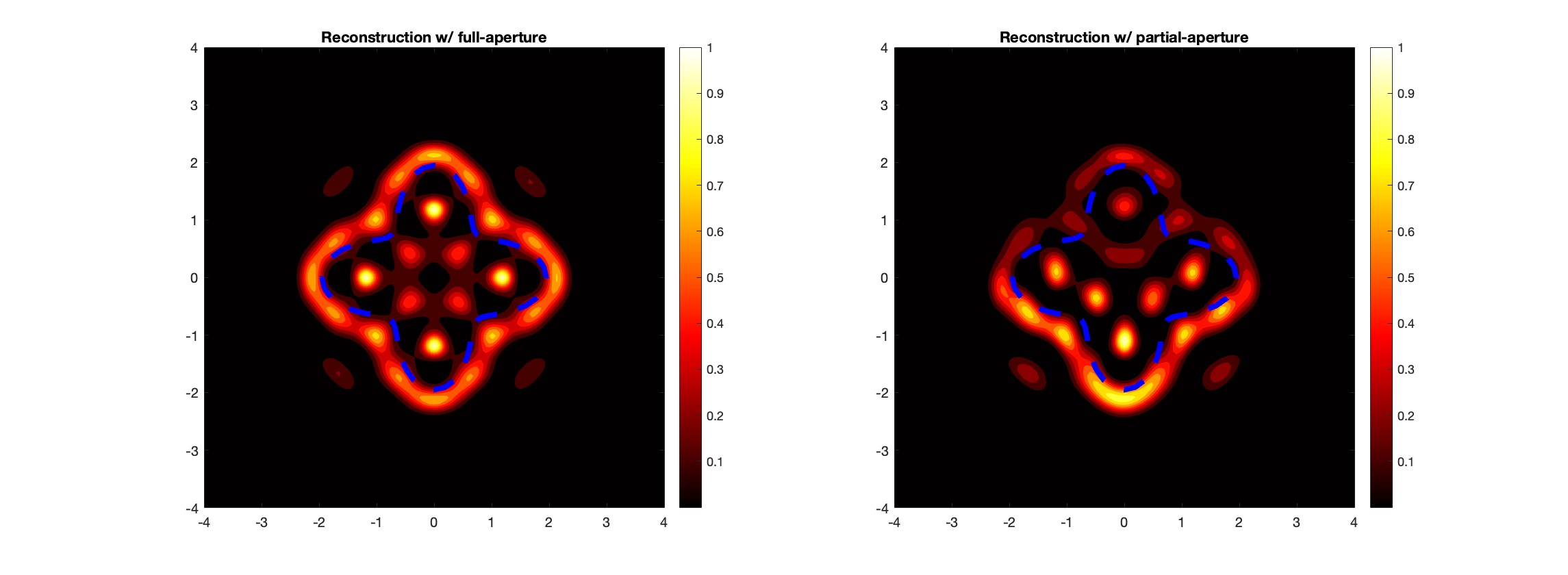}}  
\caption{The reconstruction of the star-shaped cavity with $\delta=0.1$, which corresponds to a 10$\%$ noise level via the imaging function $W_{\text{ip}}(z)$. (Left) the reconstruction with full aperture; (Right) the reconstruction with partial aperture.}
\label{recon7} 
\end{figure}

Recent studies have explored qualitative methods, such as the direct sampling method, for use with partial aperture data. For example, \cite{datacomppaper,LiuSun} investigated data completion algorithms for certain acoustic scattering problems. Additionally, \cite{partap1,partap2} derived an asymptotic formula for their imaging functional when only part of $\mathbb{S}^1$ is integrated, utilizing the Funk--Hecke integral identity. This approach was also applied to a different inverse acoustic scattering problem in \cite{partap3}, where the scatterer is a penetrable isotropic object with two conductivity parameters on its boundary.

\section{Conclusion}\label{end}

We have conducted a resolution analysis for two straightforward direct sampling imaging functions aimed at recovering an unknown clamped cavity from biharmonic far-field data. To date, there has been limited investigation into extending qualitative reconstruction methods to biharmonic scattering problems. For instance, \cite{bih-lsm} studied the linear sampling method using near-field data, which appears to be the only other work addressing an inverse shape problem for biharmonic waves using a qualitative approach. Additionally, we have provided numerical validation of our method. Our experiments demonstrate that the direct sampling imaging functions are stable and effective for handling non-convex cavities.

There are several potential directions for expanding this research. For instance, it would be valuable to investigate the applicability of the direct sampling method to cases involving penetrable cavities. Additionally, exploring the associated transmission eigenvalue problems that arise in the study of the far-field operator could offer further insights and contribute to the advancement of the field.

\noindent{\bf Acknowledgments:} The research of I. Harris and H. Lee is partially supported by the NSF DMS Grants 2107891 and 2208256. 



\end{document}